\documentclass[12pt]{amsart}
\usepackage{geometry}   
\usepackage[colorlinks,citecolor = red, linkcolor=blue,hyperindex]{hyperref}
\usepackage{euscript,verbatim, mathrsfs}
\usepackage[psamsfonts]{amssymb}
\usepackage{bbm}
\usepackage{graphicx}
\usepackage{float}
\usepackage{float, tikz}
\usepackage{extpfeil}
\usepackage[all, cmtip]{xy}
\usepackage{upref, xcolor, dsfont}
\usepackage{amsfonts,amsmath,amstext,amsbsy, amsopn,amsthm}
\usepackage{enumerate}
\usepackage{url}
\usepackage{mathtools}
\usepackage{bookmark}
\usepackage{euscript}
\usepackage{mathptmx}       
\usepackage{helvet}         
\usepackage{courier}        
\usepackage{type1cm}        
\usepackage{multicol}       
\usepackage[bottom]{footmisc}
\newtheorem{theorem}{Theorem}[section]
\newtheorem*{theorem*}{Theorem B}
\newtheorem{lemma}[theorem]{Lemma}

\newtheorem*{definition*}{Definition}
\newtheorem*{remark*}{Remark}

\newtheorem*{observation*}{Observation}

\newtheorem*{assumption*}{Assumption}
\newtheorem*{question*}{Question}
\newtheorem*{problem*}{Problem}

\newtheorem{mainthm}{Theorem}

\newtheorem{keylem}[mainthm]{Lemma}

\newcommand{\R}{\mathbb{R}}

\newcommand{\D}{\mathbb{D}}
\newcommand{\C}{\mathbb{C}}
\newcommand{\E}{\mathbb{E}}

\newcommand{\PP}{\mathbb{P}}

\newcommand{\Var}{\mathrm{Var}}

\newcommand{\Conf}{\mathrm{Conf}}

\newcommand{\an}{\text{\, and \,}}

\begin{document}
\title[LLN for vector-valued linear statistics]{A Law of large numbers for vector-valued linear statistics of Bergman DPP}

\author
{Zhaofeng Lin}
\address
{Zhaofeng Lin: Shanghai Center for Mathematical Sciences, Fudan University, Shanghai, 200438, China; Centrale Marseille, CNRS, Institut de Math{\'e}matiques de Marseille, Aix-Marseille Universit{\'e}, Marseille, 13453, France}
\email{zflin18@fudan.edu.cn, zhaofeng.LIN@univ-amu.fr}

\author
{Yanqi Qiu}
\address
{Yanqi Qiu: School of Fundamental Physics and Mathematical Sciences, HIAS, University of Chinese Academy of Sciences, Hangzhou, 310024, China}
\email{yanqi.qiu@hotmail.com, yanqiqiu@ucas.ac.cn}

\author
{Kai Wang}
\address
{Kai Wang: School of Mathematical Sciences, Fudan University, Shanghai, 200433, China}
\email{kwang@fudan.edu.cn}

\begin{abstract}
We establish a law of large numbers for a certain class of vector-valued linear statistics for the Bergman determinantal point process on the unit disk.  Our result seems to be the first LLN for vector-valued linear statistics in the setting of determinantal point processes.   As an application,  we prove that,  for almost  all configurations $X$ with respect to the Bergman determinantal point process,  the weighted Poincar\'e series (we denote by $d_{\mathrm{h}}(\cdot,\cdot)$ the hyperbolic distance on $\D$) 
\begin{align*}
\sum_{k=0}^\infty\sum_{x\in X\atop k\le d_{\mathrm{h}}(z,x)<k+1}e^{-sd_{\mathrm{h}}(z,x)}f(x)
\end{align*}
cannot be simultaneously  convergent for all Bergman functions $f\in A^2(\mathbb{D})$ whenever $1<s<3/2$. This confirms a result announced without proof in Bufetov-Qiu \cite{BQ2}.    
\end{abstract}


\maketitle

\setcounter{equation}{0}

\section{Introduction}
Let $A^2(\D)$ denote the standard Bergman space on the unit disk $\D \subset \C$ with respect to the normalized Lebesgue measure $\mathrm{d}\mu=\frac{1}{\pi}\mathrm{d}x\mathrm{d}y$: 
\begin{align*}
A^2(\D)=\Big\{f:\D\to\mathbb{C}\,\Big|\,f\,\,\text{is holomorphic and}\,\int_{\D}|f(z)|^2\mathrm{d}\mu(z)<\infty\Big\}.
\end{align*}
The Hilbert space  $A^2(\D)$ admits a reproducing kernel (called Bergman kernel of $\D$) given by 
\begin{align*}
K(z,w)=\frac{1}{(1-z\overline{w})^2},\quad z,w\in\mathbb{D}.
\end{align*}
The reader is referred to \cite{HKZ, Zh} for more details on the theory of Bergman spaces.

Let $\mathbb{P}_{K}$ be the determinantal point process on $\D$ induced by the kernel $K$. Namely, $\PP_{K}$ is a probability measure on the space $\Conf(\D)$ of locally finite configurations of $\D$ (see, e.g., \cite{Bo, BQ1, Le1, Le2, Le3, ST1, ST2, So} for more details and backgrounds). 
It is well-known  (see Peres-Vir\'ag~\cite{PV}) that this determinantal point process $\PP_K$ describes the probability distribution of the zero set of the hyperbolic Gaussian analytic function: 
\begin{align*}
\sum_{n =0}^\infty g_n z^n\,\,\,\,\,\text{with \,$g_n$\, i.i.d. standard complex Gaussian random variables}. 
\end{align*}

In Bufetov-Qiu~\cite{BQ2}, the authors studies the random  Patterson-Sullivan reconstruction of  harmonic functions from a typical configuration sampled from  the determinantal point process $\PP_K$. One of the key ideas in Bufetov-Qiu~\cite{BQ2} is a simultaneous law of large numbers for the following weighted Poincar\'e series (as the parameter $s$ approaches the critical exponent $s=1$): 
\begin{align}\label{def-PS-series}
\sum_{k=0}^\infty\sum_{x\in X\atop k\le d_{\mathrm{h}}(z,x)<k+1}e^{-sd_{\mathrm{h}}(z,x)}f(x)
\end{align}
with $f$ being any function in a certain function space (the reader is referred to \cite[\S 1.3]{BQ2} for the reason of double summations in \eqref{def-PS-series}). Here    $d_{\mathrm{h}}(\cdot,\cdot)$ denotes the hyperbolic distance on $\D$: 
\begin{align*}
d_{\mathrm{h}}(z,x)=\log\frac{1+|\varphi_z(x)|}{1-|\varphi_z(x)|}\,\,\,\,\,\text{with}\,\,\,\,\,\varphi_z(x)=\frac{z-x}{1-\overline{z}x},\quad z,x\in\mathbb{D}.
\end{align*}

It is shown in \cite[Theorem~1.1]{BQ2} that for any fixed  $f\in A^2(\D)$, the series \eqref{def-PS-series} converges for $\PP_K$-almost every configuration $X\in \Conf(\D)$. Moreover, such convergence is shown \cite[Theorem~1.2]{BQ2} to hold simultaneously for all $f$ being in a certain weighted Bergman space. However, when $s$ approaches the critical exponent $s =1$, the series \eqref{def-PS-series} cannot be defined simultaneously for all $f \in A^2(\D)$. More precisely, we have 
\begin{theorem}\label{thm-main}
For any $1<s<3/2$ and any $z\in\mathbb{D}$, we have 
\begin{align*}
\PP_K \Big(\Big\{X\in\Conf(\D)\,\Big|\,\text{$\sum_{k=0}^\infty\sum_{x\in X\atop k\le d_{\mathrm{h}}(z,x)<k+1}e^{-sd_{\mathrm{h}}(z,x)}f(x)\,\,\mathrm{converges\,\,for\,\,all}\,\,f\in A^2(\D)$}\Big\}\Big)= 0.
\end{align*}
\end{theorem}

The proof of Theorem~\ref{thm-main} relies on a new law of large numbers obtained in Theorem~\ref{thm-main-bis}  for the vector-valued linear statistics introduced below: 
\begin{align}\label{def-the-X}
\Theta_N^{(s,z)}(X):= \sum_{k=0}^{N-1} \sum_{x\in X\atop k\le d_{\mathrm{h}}(z,x)<k+1}e^{-sd_{\mathrm{h}}(z,x)}K(\cdot, x)=\sum_{x\in X\cap\mathcal{U}_N(z)}e^{-s d_{\mathrm{h}}(z,x)} K(\cdot, x), 
\end{align}
where $\mathcal{U}_N(z)$ denotes the hyperbolic disc $\mathcal{U}_N(z)=\{x\in\D:d_\mathrm{h}(z,x)<N\}$. 

For simplifying the notations, throughout the whole paper, we shall fix the probability measure $\PP_K$ on $\Conf(\D)$ and use the simplified notations: 
\begin{align*}
\E=\E_{\PP_{K}} \an \Var=\Var_{\PP_{K}}. 
\end{align*}

\begin{theorem}\label{thm-main-bis}
For any $1<s<3/2$ and any $z\in\mathbb{D}$, the sequence $\big\{\|\Theta_N^{(s,z)}\|_{A^2(\D)}^2 \big\}_{N\ge1}$ satisfies the law of large numbers, that is, 
\begin{align*}
\lim_{N\to\infty}\frac{\|\Theta_N^{(s,z)}\|_{A^2(\D)}^2}{\E\big[\|\Theta_N^{(s,z)}\|_{A^2(\D)}^2\big]}=1\quad a.s.
\end{align*}
\end{theorem}

Our proof of the law of large numbers in Theorem~\ref{thm-main-bis} is based on the following  key lemmas.

\begin{keylem}\label{lem-exp}
For any $1<s<3/2$, there exist $C_1(s),C_2(s)>0$, such that for any $z\in\D$ and sufficiently large $N$, 
\begin{align*}
C_1(s)\frac{|z|^4+4|z|^2+1}{(1-|z|^2)^2}e^{(3-2s)N}\le\mathbb{E}\big[\big\|\Theta_N^{(s,z)}\big\|_{A^2(\D)}^2\big]\le C_2(s)\frac{|z|^4+4|z|^2+1}{(1-|z|^2)^2}e^{(3-2s)N}.
\end{align*}
\end{keylem}

\begin{keylem}\label{lem-var}
For any $1<s<3/2$, there exists $C(s)>0$, such that for any $z\in\D$ and sufficiently large $N$, 
\begin{align*}
\Var\big(\|\Theta_N^{(s,z)}\|_{A^2(\D)}^2\big)\le C(s)\frac{(1+|z|)^4}{(1-|z|)^4}e^{(3-2s)N}.
\end{align*}
\end{keylem}

We end this section by providing the standard derivation of Theorem~\ref{thm-main-bis} from Lemmas~\ref{lem-exp} and \ref{lem-var}, and the standard derivation of  Theorem~\ref{thm-main} from Theorem~\ref{thm-main-bis}. 
{\flushleft \it The derivation of Theorem~\ref{thm-main-bis} from Lemmas~\ref{lem-exp} and \ref{lem-var}.}   Fix any $1<s <3/2$ and any $z\in \D$.  Lemmas~\ref{lem-exp} and \ref{lem-var} together imply that 
\begin{align*}
 \E \Bigg[  \sum_{N=1}^\infty \bigg| \frac{\|\Theta_N^{(s,z)}\|_{A^2(\D)}^2}{\E\big[\|\Theta_N^{(s,z)}\|_{A^2(\D)}^2\big]} - 1 \bigg|^2  \Bigg] =  \sum_{N=1}^\infty \frac{\Var\big(\|\Theta_N^{(s,z)}\|_{A^2(\D)}^2\big) }{\Big(\E\big[\big\|\Theta_N^{(s,z)}\big\|_{A^2(\D)}^2\big]\Big)^2} <\infty. 
\end{align*}
This clearly completes the proof of Theorem~\ref{thm-main-bis}.
{\flushleft \it The derivation of Theorem~\ref{thm-main} from Theorem~\ref{thm-main-bis}.}  Fix any $1<s <3/2$ and any $z\in \D$.  Theorem~\ref{thm-main-bis} and Lemma~\ref{lem-exp} together imply that for $\PP_K$-almost every $X\in \Conf(\D)$,
\begin{align*}
\lim_{N\to\infty} \| \Theta_N^{(s,z)} (X) \|_{A^2(\D)} = \infty.  
\end{align*}
Note that, if  $X\in \Conf(\D)$  is such that the series \eqref{def-PS-series} converges for all $f\in A^2(\D)$,  then by the uniform boundedness principle (see, e.g., \cite[Theorem~2.5]{Ru}), we must have 
\begin{align*}
\sup_{N\ge 1} \| \Theta_N^{(s,z)} (X) \|_{A^2(\D)} <\infty. 
\end{align*} 
This clearly completes the proof of Theorem~\ref{thm-main}.

{\flushleft\bf Acknowledgements.} ZL's work is supported by  the China Scholarship Council (202106100127).
YQ's work is supported by the National Natural Science Foundation of China (No.12288201).
KW's work is supported by the National Natural Science Foundation of China (No.12231005, No.12326376) and the Shanghai Technology Innovation Project (21JC1400800).

\section{The proof of Lemma~\ref{lem-exp}}
From now on, we shall always fix $1<s<3/2$ and $z\in\mathbb{D}$. Recall the definition \eqref{def-the-X} of the vector-valued linear statistics  $\Theta_N^{(s,z)}(X)$.  For any integer $N\ge1$ and any configuration $X\in \Conf(\D)$, we set
\begin{align*}
S_N(X)=S^{(s,z)}_N(X):=\|\Theta_N^{(s,z)}(X)\|_{A^2(\D)}^2.
\end{align*}
Using the reproducing property of the Bergman kernel
\begin{align*}
\big\langle K(\cdot,x_1),K(\cdot,x_2)\big\rangle_{A^2(\D)}=K(x_2, x_1),
\end{align*}
we have 
\begin{align}\label{def-snx}
S_N(X)=\sum_{\substack{x_1,\,x_2\in X\cap\mathcal{U}_N(z)}}e^{-sd_{\mathrm{h}}(z,x_1)-sd_{\mathrm{h}}(z,x_2)}K(x_1,x_2).
\end{align}
As usual, for using the explicit correlation functions (due to the determinantal structure) of our point process $\PP_{K}$, we decompose $S_N$ as 
\begin{align*}
S_N(X)=\underbrace{\sum_{\substack{x\in X\cap\mathcal{U}_N(z)}}e^{-2 sd_{\mathrm{h}}(z,x)}K(x,x) }_{\text{denoted $S_{N,1}(X)$}}+\underbrace{\sum_{\substack{x_1,\,x_2\in X\cap\mathcal{U}_N(z)\\x_1\ne x_2}}e^{-sd_{\mathrm{h}}(z,x_1)-sd_{\mathrm{h}}(z,x_2)}K(x_1,x_2)}_{\text{denoted $S_{N,2}(X)$}},
\end{align*}
and hence
\begin{align}\label{expression1}
\mathbb{E} [S_N]=\mathbb{E}[S_{N,1}]+\mathbb{E}[S_{N,2}].
\end{align}

\subsection{The estimate for $\E[S_{N,1}]$}
\begin{lemma}\label{lem-Sn1}
We have 
\begin{align*}
	\lim_{N\to\infty} \frac{\E[S_{N,1}]}{e^{(3-2s)N}} = \frac{|z|^4+4|z|^2+1}{2^6(3-2s)(1-|z|^2)^2}.  
\end{align*}
\end{lemma}

\begin{proof}
By the definition of Bergman determinantal point process $\PP_{K}$, we have
\begin{align*}
		\mathbb{E}[S_{N,1}]=\int_{\mathcal{U}_N(z)}e^{-2sd_{\mathrm{h}}(z,x)}K(x,x)^2\mathrm{d}\mu(x).
\end{align*}
We introduce the following notations: 
\begin{align}\label{def-rn-Tz}
		r_N=\frac{e^N-1}{e^N+1} \an T_z(x)=  e^{-d_{\mathrm{h}}(z, x)} =  \frac{1-|\varphi_z(x)|}{1+|\varphi_z(x)|}. 
\end{align}
Then $\mathcal{U}_N(z)=\{x\in\D:|\varphi_z(x)|<r_N\}$ and 
\begin{align*}
		\mathbb{E}[S_{N,1}]=\int_{|\varphi_z(x)|<r_N}T_z(x)^{2s}K(x,x)^2\mathrm{d}\mu(x).
\end{align*}
We are going to make a change-of-variable $w=\varphi_z(x)$, then 
\begin{align*}
		x=\varphi_z(w) \an \mathrm{d}\mu(x)=\frac{(1-|z|^2)^2}{|1-\overline{z}w|^4}\mathrm{d}\mu(w).
\end{align*}
Using the standard equality (see, e.g., \cite{HKZ, Zh})
\begin{align*}
		\frac{1}{1-|\varphi_z(w)|^2}=\frac{|1-\overline{z}w|^2}{(1-|z|^2)(1-|w|^2)}, 
\end{align*}
we obtain 
\begin{align*}
		\mathbb{E}[S_{N,1}]=\frac{1}{(1-|z|^2)^2}\int_{D(0, r_N)}\frac{(1-|w|)^{2s-4}}{(1+|w|)^{2s+4}}(1-\overline{z}w)^2(1-z\overline{w})^2\mathrm{d}\mu(w),
\end{align*}
where $D(0, r_N) = \{w\in \D: |w|<r_N\}$.  Under the polar coordinates, we get
\begin{align*}
	\mathbb{E}[S_{N,1}]=\frac{2}{(1-|z|^2)^2}\int_{0}^{r_N}\frac{(1-r)^{2s-4}}{(1+r)^{2s+4}}r\big(|z|^4r^4+4|z|^2r^2+1\big)\mathrm{d}r.
\end{align*}
Since $1<s<3/2$,  we have 
\begin{align*}
		\lim_{u\to1^-}\frac{\int_{0}^{u}\frac{(1-r)^{2s-4}}{(1+r)^{2s+4}}r(|z|^4r^4+4|z|^2r^2+1)\mathrm{d}r}{\frac{|z|^4+4|z|^2+1}{2^{2s+4}(3-2s)}(1-u)^{2s-3}}=1.
\end{align*}
Consequently, when $N\to\infty$, we have
\begin{align*}
		\mathbb{E}[S_{N,1}]\sim\frac{2}{(1-|z|^2)^2}\frac{|z|^4+4|z|^2+1}{2^{2s+4}(3-2s)}(1-r_N)^{2s-3}=\frac{|z|^4+4|z|^2+1}{2^6(3-2s)(1-|z|^2)^2}(e^N+1)^{3-2s}.
\end{align*}

This completes the proof of Lemma~\ref{lem-Sn1}.
\end{proof}

\subsection{The estimate for $\E[S_{N,2}]$}
\begin{lemma}\label{lem-Sn2}
There exist $B(s,z), C(s,z)\in \R$, such that 
for large enough $N$, we have 
\begin{align*}
		-\frac{5\cdot2^{s-9}(|z|^4+4|z|^2+1)}{(4-s)(3-2s)(1-|z|^2)^2}(e^N+1)^{3-2s}+B(s,z)\leq\E[S_{N,2}]\leq C(s,z).
\end{align*}
\end{lemma}

\begin{remark*}
In Lemma \ref{lem-Sn2}, we do not obtain a correct asymptotic order of  $\E[S_{N,2}]$.  However,  the above two-sided estimates of  $\E[S_{N,2}]$, combined with the asymptotic order of  $\E[S_{N,1}]$ obtained in Lemma~\ref{lem-Sn1},  is sufficient for us to obtain the correct asymptotic order of $\E[S_N] = \E[S_{N,1}] + \E[S_{N,2}]$. See \S \ref{sec-proofA} below.
\end{remark*}

Recall that  the  correlation functions of the determinantal point process $\PP_{K}$ are given by 
\begin{align*}
		\rho_k(x_1,\cdots,x_k) = \det[K(x_i, x_j)]_{1\le i, j \le k}, \quad k\ge 1. 
\end{align*}
More precisely, for any bounded measurable compactly supported function $\phi:\mathbb{D}\to\mathbb{C}$, we have
\begin{align}\label{def-corr}
	\mathbb{E}\Big[\mathop{\sum\nolimits^\sharp}_{x_1,\,\cdots,\,x_k\in X}\phi(x_1,\,\cdots,\,x_k)\Big]=\int_{\mathbb{D}^k}\phi(x_1,\,\cdots,\,x_k)\det[K(x_i, x_j)]_{1\le i, j \le k}\prod_{j=1}^{k}\mu(x_j),
\end{align}
here and below we use the notation:
\begin{align}\label{def-sum-star}
	\mathop{\sum\nolimits^\sharp}_{x_1,\,\cdots,\,x_k\in X}=\sum_{\substack{x_1,\,\cdots,\,x_k\in X\\x_1,\,\cdots,\,x_k\text{ distinct}}}.
\end{align}
In particular, 
\begin{align*}
		\rho_2(x_1, x_2) = \det[K(x_i, x_j)]_{1\le i, j \le 2} = K(x_1, x_1) K(x_2, x_2) -  K(x_1, x_2) K(x_2, x_1). 
\end{align*}
Therefore, using the notations \eqref{def-rn-Tz},  we obtain 
\begin{align*}
		 \mathbb{E}[S_{N,2}] = &  \E\Big[ \mathop{\sum\nolimits^\sharp}_{x_1,\,x_2\in X\cap\mathcal{U}_N(z)}  T_z(x_1)^s T_z(x_2)^s K(x_1,x_2) \Big]
		 \\
		 = &  	 \int_{\mathcal{U}_N(z)^2} T_z(x_1)^s T_z(x_2)^s  K(x_1,x_2)  \rho_2(x_1, x_2) \mathrm{d}\mu(x_1)\mathrm{d}\mu(x_2).
\end{align*}
Hence $\E[S_{N,2}]$ has the following decomposition: 
\begin{align}\label{expression2}
		\E[S_{N,2}] = I_{N,1} - I_{N,2}
\end{align}
with
\begin{align*}
		I_{N,1}=\int_{\mathcal{U}_N(z)^2}T_z(x_1)^sT_z(x_2)^sK(x_1,x_1)K(x_2,x_2)K(x_1,x_2)\mathrm{d}\mu(x_1)\mathrm{d}\mu(x_2)
\end{align*}
and
\begin{align*}
		I_{N,2}=\int_{\mathcal{U}_N(z)^2}T_z(x_1)^sT_z(x_2)^sK(x_1,x_2)^2  K(x_2,x_1) \mathrm{d}\mu(x_1)\mathrm{d}\mu(x_2).
\end{align*}

\subsubsection{The estimate for $I_{N,1}$}
\begin{lemma}\label{lem-I1N}
We have 
\begin{align*}
\lim_{N\to\infty}I_{N,1}=\frac{4}{(1-|z|^2)^2}\Big(\int_{0}^{1}\frac{(1-r)^{s-2}}{(1+r)^{s+2}}r\mathrm{d}r\Big)^2.
\end{align*}
\end{lemma}

\begin{proof}
For $k=1,2$, make a change-of-variable $w_k=\varphi_z(x_k)$, then 
\begin{align*}
		x_k=\varphi_z(w_k) \an \mathrm{d}\mu(x_k)=\frac{(1-|z|^2)^2}{|1-\overline{z}w_k|^4}\mathrm{d}\mu(w_k).
\end{align*}
Using the standard equalities (see, e.g., \cite{HKZ, Zh})
\begin{align*}
		\frac{1}{1-|\varphi_z(w_k)|^2}=\frac{|1-\overline{z}w_k|^2}{(1-|z|^2)(1-|w_k|^2)}
\end{align*}
and
\begin{align*}
		\frac{1}{1-\varphi_z(w_1)\overline{\varphi_z(w_2)}}=\frac{(1-\overline{z}w_1)(1-z\overline{w}_2)}{(1-|z|^2)(1-w_1\overline{w}_2)},
\end{align*}
we obtain 
\begin{align*}
		I_{N,1}=\frac{1}{(1-|z|^2)^2}\int_{D(0, r_N)^2}\frac{(1-|w_1|)^{s-2}}{(1+|w_1|)^{s+2}}\frac{(1-|w_2|)^{s-2}}{(1+|w_2|)^{s+2}}\frac{(1-\overline{z}w_1)^2(1-z\overline{w}_2)^2}{(1-w_1\overline{w}_2)^2}\mathrm{d}\mu(w_1)\mathrm{d}\mu(w_2).
\end{align*}
Under the polar coordinates, we get
\begin{align*}
		I_{N,1}=\frac{4}{(1-|z|^2)^2}\Big(\int_{0}^{r_N}\frac{(1-r)^{s-2}}{(1+r)^{s+2}}r\mathrm{d}r\Big)^2.
\end{align*}
Therefore, under  the assumption $1<s < 3/2$, we obtain
\begin{align*}
		\lim_{N\to\infty}I_{N,1}=\frac{4}{(1-|z|^2)^2}\Big(\int_{0}^{1}\frac{(1-r)^{s-2}}{(1+r)^{s+2}}r\mathrm{d}r\Big)^2<\infty.
\end{align*}

This completes the proof of Lemma~\ref{lem-I1N}.
\end{proof}

\subsubsection{The estimate for $I_{N,2}$}
\begin{lemma}\label{lem-I2N}
There exists $A(s,z)\in \R$, such that for large enough $N$, we have 
\begin{align*}
0\le I_{N,2}\leq\frac{5\cdot2^{s-9}(|z|^4+4|z|^2+1)}{(4-s)(3-2s)(1-|z|^2)^2}(e^N+1)^{3-2s}+\frac{16(|z|^4+4|z|^2+1)}{(1-|z|^2)^2}A(s,z).
\end{align*}
\end{lemma}

\begin{proof}
Similar to the proof of Lemma~\ref{lem-I1N}, by making a change-of-variable $w_k=\varphi_z(x_k)$, $k=1,2$, we get
\begin{align*}
		I_{N,2}=\frac{1}{(1-|z|^2)^2}\int_{D(0, r_N)^2}\Big(\frac{1-|w_1|}{1+|w_1|}\Big)^s\Big(\frac{1-|w_2|}{1+|w_2|}\Big)^s \frac{(1-\overline{z}w_1)^2(1-z\overline{w}_2)^2}{(1-w_1\overline{w}_2)^4(1-\overline{w}_1w_2)^2}\mathrm{d}\mu(w_1)\mathrm{d}\mu(w_2).
\end{align*}
Then, under the polar coordinates, we have 
\begin{align*}
		I_{N,2}=\frac{4}{(1-|z|^2)^2}\int_{0}^{r_N}\int_{0}^{r_N}\Big(\frac{1-r_1}{1+r_1}\Big)^s\Big(\frac{1-r_2}{1+r_2}\Big)^s\frac{r_1r_2}{(1-r_1^2r_2^2)^5}\mathcal{R}(r_1,r_2)\mathrm{d}r_1\mathrm{d}r_2,
\end{align*}
where
\begin{align*}
		\mathcal{R}(r_1,r_2)=|z|^4r_1^6r_2^6+(3|z|^4+8|z|^2)r_1^4r_2^4+(8|z|^2+3)r_1^2r_2^2+1.
\end{align*}
It follows that 
\begin{align*}
		0\leq I_{N,2}\leq\frac{16(|z|^4+4|z|^2+1)}{(1-|z|^2)^2}\int_{0}^{r_N}\Big(\frac{1-r_2}{1+r_2}\Big)^s\int_{0}^{r_N}\Big(\frac{1-r_1}{1+r_1}\Big)^s\frac{1}{(1-r_1^2r_2^2)^5}\mathrm{d}r_1\mathrm{d}r_2.
\end{align*}
Let
\begin{align*}
		F(u)=\int_{0}^{u}\Big(\frac{1-r_2}{1+r_2}\Big)^s\int_{0}^{u}\Big(\frac{1-r_1}{1+r_1}\Big)^s\frac{1}{(1-r_1^2r_2^2)^5}\mathrm{d}r_1\mathrm{d}r_2,\quad u\in(0,1),
\end{align*}
then
\begin{align*}
		F'(u)=2\Big(\frac{1-u}{1+u}\Big)^s\int_{0}^{u}\Big(\frac{1-r}{1+r}\Big)^s\frac{1}{(1-u^2r^2)^5}\mathrm{d}r,
\end{align*}
and make a change-of-variable $t=ur$, we get
\begin{align*}
		F'(u)=\frac{2}{u}\Big(\frac{1-u}{1+u}\Big)^s\int_{0}^{u^2}\Big(\frac{1-\frac{t}{u}}{1+\frac{t}{u}}\Big)^s\frac{1}{(1-t^2)^5}\mathrm{d}t\leq\frac{2}{u}\Big(\frac{1-u}{1+u}\Big)^s\int_{0}^{u^2}\Big(\frac{1-t}{1+t}\Big)^s\frac{1}{(1-t^2)^5}\mathrm{d}t.
\end{align*}
Since $1<s<3/2$, we have
\begin{align*}
		\lim_{u\to1^-}\frac{\frac{2}{u}(\frac{1-u}{1+u})^s\int_{0}^{u^2}(\frac{1-t}{1+t})^s\frac{1}{(1-t^2)^5}\mathrm{d}t}{\frac{(1-u)^{2s-4}}{2^{s+8}(4-s)}}=1.
\end{align*}
Therefore, there exists $u_0\in(0,1)$, such that for any $u\in(u_0,1)$,
\begin{align*}
		F'(u)\leq\Big(1+\frac{1}{4}\Big)\frac{(1-u)^{2s-4}}{2^{s+8}(4-s)}=\frac{5(1-u)^{2s-4}}{2^{s+10}(4-s)}. 
\end{align*}
Hence for any $u\in (u_0, 1)$, 
\begin{align*}
		F(u)=\int_{u_0}^{u}F'(t)\mathrm{d}t+F(u_0)\leq\frac{5(1-u)^{2s-3}}{2^{s+10}(4-s)(3-2s)}+A(s,z),
\end{align*}
where
\begin{align*}
		A(s,z)=-\frac{5(1-u_0)^{2s-3}}{2^{s+10}(4-s)(3-2s)}+F(u_0).
\end{align*}
Consequently, for large enough $N$ satisfying $r_N>u_0$, we have
\begin{align*}
		I_{N,2}\leq\frac{16(|z|^4+4|z|^2+1)}{(1-|z|^2)^2}\Big[\frac{5(1-r_N)^{2s-3}}{2^{s+10}(4-s)(3-2s)}+A(s,z)\Big].
\end{align*}
That is, 
\begin{align*}
		0\leq I_{N,2}\leq\frac{5\cdot2^{s-9}(|z|^4+4|z|^2+1)}{(4-s)(3-2s)(1-|z|^2)^2}(e^N+1)^{3-2s}+\frac{16(|z|^4+4|z|^2+1)}{(1-|z|^2)^2}A(s,z).
\end{align*}

This completes the proof of Lemma~\ref{lem-I2N}.
\end{proof}

\subsubsection{The conclusion for $\E[S_{N,2}]$: the proof of Lemma~\ref{lem-Sn2}}
\begin{proof}[Proof of Lemma~\ref{lem-Sn2}]
By Lemma~\ref{lem-I1N},  when $N$ is large enough, we have
\begin{align}\label{est-exp2}
		\frac{3}{(1-|z|^2)^2}\Big(\int_{0}^{1}\frac{(1-r)^{s-2}}{(1+r)^{s+2}}r\mathrm{d}r\Big)^2\leq I_{N,1}\leq\frac{5}{(1-|z|^2)^2}\Big(\int_{0}^{1}\frac{(1-r)^{s-2}}{(1+r)^{s+2}}r\mathrm{d}r\Big)^2.
\end{align}
Combine \eqref{expression2} with \eqref{est-exp2} and Lemma~\ref{lem-I2N}, for  large enough $N$, we have 
\begin{align*}
		-\frac{5\cdot2^{s-9}(|z|^4+4|z|^2+1)}{(4-s)(3-2s)(1-|z|^2)^2}(e^N+1)^{3-2s}+B(s,z)\leq\E[S_{N,2}]\leq C(s,z),
\end{align*}
where
\begin{align*}
		B(s,z)=\frac{3}{(1-|z|^2)^2}\Big(\int_{0}^{1}\frac{(1-r)^{s-2}}{(1+r)^{s+2}}r\mathrm{d}r\Big)^2-\frac{16(|z|^4+4|z|^2+1)}{(1-|z|^2)^2}A(s,z)
\end{align*}
and
\begin{align*}
		C(s,z)=\frac{5}{(1-|z|^2)^2}\Big(\int_{0}^{1}\frac{(1-r)^{s-2}}{(1+r)^{s+2}}r\mathrm{d}r\Big)^2.
\end{align*}

This completes the proof of Lemma~\ref{lem-Sn2}.
\end{proof}

\subsection{The conclusion for $\E[S_{N}]$: the proof of Lemma~\ref{lem-exp}}\label{sec-proofA}
\begin{proof}[Proof of Lemma~\ref{lem-exp}]
By Lemma~\ref{lem-Sn1},  when $N$ is large enough, we have
\begin{align}\label{est-exp1}
		\frac{3(|z|^4+4|z|^2+1)}{2^8(3-2s)(1-|z|^2)^2}(e^N+1)^{3-2s}\leq  \mathbb{E}[S_{N,1}]\leq\frac{5(|z|^4+4|z|^2+1)}{2^8(3-2s)(1-|z|^2)^2}(e^N+1)^{3-2s}.
\end{align}
By \eqref{expression1}, \eqref{est-exp1} and Lemma~\ref{lem-Sn2}, when $N$ is large enough,
\begin{align*}
		\mathbb{E}[S_N]\leq\frac{5(|z|^4+4|z|^2+1)}{2^8(3-2s)(1-|z|^2)^2}(e^N+1)^{3-2s}+C(s,z)
\end{align*}
and
\begin{align*}
		\mathbb{E}[S_N]&\geq\frac{3(|z|^4+4|z|^2+1)}{2^8(3-2s)(1-|z|^2)^2}(e^N+1)^{3-2s}-\frac{5\cdot2^{s-9}(|z|^4+4|z|^2+1)}{(4-s)(3-2s)(1-|z|^2)^2}(e^N+1)^{3-2s}+B(s,z)\\
		&=\frac{(12-3s-5\cdot2^{s-1})(|z|^4+4|z|^2+1)}{2^8(4-s)(3-2s)(1-|z|^2)^2}(e^N+1)^{3-2s}+B(s,z).
\end{align*}
Finally, it is easy to check that $12-3s-5\cdot2^{s-1}>0$ for any $1<s<3/2$. Thus for large enough $N$, we have
\begin{align*}
		\frac{(12-3s-5\cdot2^{s-1})(|z|^4+4|z|^2+1)}{2^9(4-s)(3-2s)(1-|z|^2)^2}e^{(3-2s)N}\leq\mathbb{E}[S_N]\leq\frac{5(|z|^4+4|z|^2+1)}{2^7(3-2s)(1-|z|^2)^2}e^{(3-2s)N}.
\end{align*}

This completes the whole proof of Lemma~\ref{lem-exp}.
\end{proof}

\section{The proof of Lemma~\ref{lem-var}}
In this section, we are going to estimate the variance
\begin{align*}
	\Var(S_N)=\mathbb{E}[S_N^2]-\big(\mathbb{E}[S_N]\big)^2.
\end{align*}
In \S \ref{sec-es2} and \S \ref{sec-square} respectively, 
we are going to decompose both  terms $\mathbb{E}[S_N^2]$ and $\big(\mathbb{E}[S_N]\big)^2$ into sums of integrals of different orders.     Then, in \S \ref{sec-var}, we regroup these integrals  in a suitable way into the formula \eqref{dec-var} for $\Var(S_N)$, some major cancellations among these integrals will be used in the estimate of $\Var(S_N)$. See the discusssion at the end of \S \ref{sec-var}. 

Recall that we have fixed $1<s<3/2$ and $z\in\D$, and we will use the notations \eqref{def-rn-Tz}:
\begin{align*}
r_N=\frac{e^N-1}{e^N+1} \an  T_z(x)=  e^{-d_{\mathrm{h}}(z, x)} =  \frac{1-|\varphi_z(x)|}{1+|\varphi_z(x)|}. 
\end{align*}

\subsection{The expression for $\mathbb{E}[S_N^2]$}\label{sec-es2}
Recall the formula \eqref{def-snx} for $S_N(X)$. Then, under the notations \eqref{def-rn-Tz}, 
\begin{align*}
	S_N(X)^2=\sum_{x_1,\,x_2,\,x_3,\,x_4\in X\cap\mathcal{U}_N(z)}\Big( \prod_{j=1}^{4}T_z(x_j)^s \Big) K(x_1, x_2) K(x_3, x_4).
\end{align*}
We will need  the following decomposition:
\begin{align}\label{dec-sn}
S_N(X)^2=\sum_{k=1}^{4}L_{N,k}(X)
\end{align}
with
\begin{align*}
L_{N,k}(X)=\sum_{\substack{x_1,\,x_2,\,x_3,\,x_4\in X\cap\mathcal{U}_N(z)\\|\{x_1,\,x_2,\,x_3,\,x_4\}|=k}} \Big( \prod_{j=1}^{4}T_z(x_j)^s \Big) K(x_1, x_2) K(x_3, x_4)  ,\quad k=1,2,3,4.
\end{align*}

\subsubsection{The expressions of $L_{N,k}$}
In what follows, for simplifying the notation, we write 
\begin{align*}
K_{ij}=K(x_i,x_j), \quad 1\le i, j \le 4. 
\end{align*}
In order to use the defining formula \eqref{def-corr} of the correlation functions of  our determinantal point process $\mathbb{P}_K$, we need to decompose all terms $L_{N,k}$ into summations of the form \eqref{def-sum-star}: 
\begin{align*}
	\mathop{\sum\nolimits^\sharp}_{x_1,\,\cdots,\,x_k\in X \cap \mathcal{U}_N(z) }=\sum_{\substack{x_1,\,\cdots,\,x_k\in X \cap \mathcal{U}_N(z) \\x_1,\,\cdots,\,x_k\text{ distinct}}}.
\end{align*}

{ \it 1. The expressions of $L_{N,1}$ and $L_{N,4}$.}
 It is easy to see 
\begin{align*}
L_{N,1}(X)=\sum_{x_1\in X\cap\mathcal{U}_N(z)}T_z(x_1)^{4s}K_{11}^2
\end{align*}
and 
\begin{align*}
L_{N,4}(X)=\mathop{\sum\nolimits^\sharp}_{x_1,\,x_2,\,x_3,\,x_4\in X\cap\mathcal{U}_N(z)} \Big( \prod_{j=1}^4 T_z(x_j)^{s}\Big)K_{12}K_{34}.
\end{align*}

{ \it 2. The expression of $L_{N,2}$.}
Note first that 
\begin{align*}
L_{N,2}(X) = \sum_{\substack{x_1,\,\cdots,\,x_4 \in X \cap \mathcal{U}_N(z) \\ |\{x_1,\,\cdots,\,x_4\} | = 2}}  =    \sum_{\substack{(x_1,\,\cdots,\,x_4) \in  (X \cap \mathcal{U}_N(z))^4 \\ |\{x_1,\,\cdots,\,x_4\} | = 2}}. 
\end{align*}
Therefore, we have  the following decomposition:
\begin{align*}
L_{N,2}(X) =&   \sum_{\substack{(x_1,\,\cdots,\,x_4) \in (X \cap \mathcal{U}_N(z))^4 \\ x_1=x_2=x_3\neq x_4}} + \sum_{\substack{(x_1,\,\cdots,\,x_4) \,\in (X \cap \mathcal{U}_N(z))^4 \\ x_1=x_2=x_4\neq x_3}} + \sum_{\substack{(x_1,\,\cdots,\, x_4) \in (X \cap \mathcal{U}_N(z))^4 \\ x_1=x_3=x_4\neq x_2}}+ \sum_{\substack{(x_1,\,\cdots,\, x_4)\in (X \cap \mathcal{U}_N(z))^4 \\ x_2=x_3=x_4\neq x_1}}
 \\
  &+ \sum_{\substack{(x_1,\,\cdots,\,x_4) \in  (X \cap \mathcal{U}_N(z))^4 \\ x_1=x_2\neq x_3=x_4}} + \sum_{\substack{(x_1,\,\cdots,\, x_4) \in (X \cap \mathcal{U}_N(z))^4 \\ x_1=x_3\neq x_2=x_4}} + \sum_{\substack{(x_1,\,\cdots,\,x_4) \in (X \cap \mathcal{U}_N(z))^4 \\ x_1=x_4\neq x_2=x_3}}
\end{align*}
and thus obtain 
\begin{align}\label{dec-ln2}
\begin{split}
		L_{N,2}(X)&=\underbrace{2\mathop{\sum\nolimits^\sharp}_{x_1,\,x_2\in X\cap\mathcal{U}_N(z)}T_z(x_1)^{3s}T_z(x_2)^{s}[K_{11}K_{12}+K_{11}K_{21}]}_{\text{denoted $L_{N,2}'(X)$}}\\
		&\quad+ \underbrace{ \mathop{\sum\nolimits^\sharp}_{x_1,\,x_2\in X\cap\mathcal{U}_N(z)}T_z(x_1)^{2s}T_z(x_2)^{2s}[K_{11}K_{22}+K_{12}^2+K_{12}K_{21}]}_{\text{denoted $L_{N,2}''(X)$}}.
		\end{split}
\end{align}

{ \it 3. The expression of $L_{N,3}$.}
Similar to the above decomposition, we  write 
\begin{align*}
	L_{N,3}(X)&=  \sum_{\substack{(x_1,\,\cdots,\,x_4 )\in (X \cap \mathcal{U}_N(z) )^4\\x_1,\,x_3,\,x_4\text{ distinct}\\x_1=x_2}} + \sum_{\substack{(x_1,\,\cdots,\,x_4 )\in (X \cap \mathcal{U}_N(z) )^4 \\x_1,\,x_2,\,x_4\text{ distinct}\\x_1=x_3}}+ \sum_{\substack{(x_1,\,\cdots,\,x_4 )\in (X \cap \mathcal{U}_N(z) )^4 \\x_1,\,x_2,\,x_3\text{ distinct}\\x_1=x_4}}
	\\
	&\quad+ \sum_{\substack{(x_1,\,\cdots,\,x_4 )\in (X \cap \mathcal{U}_N(z) )^4 \\x_1,\,x_2,\,x_4\text{ distinct}\\x_2=x_3}} + \sum_{\substack{(x_1,\,\cdots,\,x_4 )\in (X \cap \mathcal{U}_N(z) )^4 \\x_1,\,x_2,\,x_3\text{ distinct}\\x_2=x_4}} + \sum_{\substack{(x_1,\,\cdots,\,x_4 )\in (X \cap \mathcal{U}_N(z) )^4 \\x_1,\,x_2,\,x_3\text{ distinct}\\x_3=x_4}}
\end{align*}
and thus obtain
\begin{align}\label{dec-ln3}
\begin{split}
		L_{N,3}(X)&=\underbrace{\mathop{\sum\nolimits^\sharp}_{x_1,\,x_2,\,x_3\in X\cap\mathcal{U}_N(z)}T_z(x_1)^{2s}T_z(x_2)^{s}T_z(x_3)^{s}[K_{12}K_{13}+K_{21}K_{31}+2K_{12}K_{31}]}_{\text{denoted $L_{N,3}'(X)$}}\\
		&\quad+  \underbrace{2\mathop{\sum\nolimits^\sharp}_{x_1,\,x_2,\,x_3\in X\cap\mathcal{U}_N(z)}T_z(x_1)^{2s}T_z(x_2)^{s}T_z(x_3)^{s}K_{11}K_{23}}_{\text{denoted $L_{N,3}''(X)$}}.
\end{split}
\end{align}

\subsubsection{The decomposition of $\mathbb{E}[S_N^2]$}\label{sec-square-J}
Using the decompositions \eqref{dec-sn}, \eqref{dec-ln2} and \eqref{dec-ln3},  we can write 
\begin{align*}
S_N(X)^2 = L_{N,1}(X) + L_{N,2}'(X) + L_{N,2}''(X)  + L_{N,3}'(X)+ L_{N,3}''(X) + L_{N,4}(X). 
\end{align*}
Therefore,  we have the following expression:
\begin{align}\label{dec-E-sn2}
		\mathbb{E}[S_N^2]=  J_{N,1} + J_{N,2}' + J_{N,2}'' + J_{N,3}'+ J_{N,3}'' + J_{N,4}
\end{align}
with 
\begin{align*}
J_{N,k} = \E[L_{N,k}],  \, \,\, J_{N,k}'=  \E[L_{N,k}']  \an   J_{N,k}''=  \E[L_{N,k}''] . 
\end{align*}

Note that  for each summand $J_{N,k}$, $J_{N,k}'$, $J_{N,k}''$ in \eqref{dec-E-sn2}, we can now use the defining formula \eqref{def-corr} of the correlation functions.   For instance, we have 
\begin{align*}
		J_{N,1}=\int_{\mathcal{U}_N(z)}T_z(x_1)^{4s} K_{11}^3 \mathrm{d}\mu(x_1),
\end{align*}
\begin{align*}
		J_{N,2}'=2\int_{\mathcal{U}_N(z)^2}T_z(x_1)^{3s}T_z(x_2)^{s}  [K_{11}K_{12}+K_{11}K_{21}]\det[K_{ij}]_{1\leq i,j\leq2}\prod_{k=1}^{2}\mathrm{d}\mu(x_k)
\end{align*}
and
\begin{align*}
	J_{N,3}'=\int_{\mathcal{U}_N(z)^3}T_z(x_1)^{2s}T_z(x_2)^{s}T_z(x_3)^{s}[K_{12}K_{13}+K_{21}K_{31}+2K_{12}K_{31}]\det[K_{ij}]_{1\leq i,j\leq3}\prod_{k=1}^{3}\mathrm{d}\mu(x_k).
\end{align*}
The other terms $J_{N,2}''$, $J_{N,3}''$ and $J_{N,4}$ also have explicit integral forms.

\subsection{The expression for $\big(\mathbb{E}[S_N]\big)^2$}\label{sec-square}
From the definition of Bergman determinantal point process $\PP_{K}$, we have
\begin{align*}
		\mathbb{E}[S_N]= \underbrace{\int_{\mathcal{U}_N(z)}T_z(x_1)^{2s}K_{11}^2\mathrm{d}\mu(x_1)}_{\text{denoted $Q_{N,1}$}}+ \underbrace{\int_{\mathcal{U}_N(z)^2}T_z(x_1)^{s}T_z(x_2)^{s}K_{12}\det[K_{ij}]_{1\leq i,j\leq2}\prod_{k=1}^{2}\mathrm{d}\mu(x_k)}_{\text{denoted $Q_{N,2}$}}.
\end{align*}
Therefore, we can decompose $\big(\mathbb{E}[S_N]\big)^2$ as
\begin{align}\label{Esn-dec}
		\big(\mathbb{E}[S_N]\big)^2=\sum_{k=1}^{3}V_{N,k}
\end{align}
with
\begin{align*}
		V_{N,1}=  Q_{N,1}^2 = \int_{\mathcal{U}_N(z)^2}T_z(x_1)^{2s}T_z(x_2)^{2s}K_{11}^2K_{22}^2\prod_{k=1}^{2}\mathrm{d}\mu(x_k),
\end{align*}
\begin{align*}
		V_{N,2}= 2 Q_{N,1} Q_{N,2} = 2\int_{\mathcal{U}_N(z)^3}T_z(x_1)^{2s}T_z(x_2)^{s}T_z(x_3)^{s}K_{11}^2K_{23}\det[K_{ij}]_{2\leq i,j\leq3}\prod_{k=1}^{3}\mathrm{d}\mu(x_k)
\end{align*}
and
\begin{align*}
		V_{N,3}= Q_{N,2}^2 = \int_{\mathcal{U}_N(z)^4}\Big(\prod_{k=1}^{4}T_z(x_k)^{s}\Big)K_{12}K_{34}\det[K_{ij}]_{1\leq i,j\leq2}\det[K_{ij}]_{3\leq i,j\leq4}\prod_{k=1}^{4}\mathrm{d}\mu(x_k).
\end{align*}

\subsection{The expression for $\Var(S_N)$}\label{sec-var}
By \eqref{dec-E-sn2} and \eqref{Esn-dec}, we have  
\begin{align}\label{dec-var}
	\begin{split}
	\Var(S_N)& =\mathbb{E}[S_N^2]-\big(\mathbb{E}[S_N]\big)^2
	\\
	&  =   J_{N,1} + J_{N,2}' + J_{N,2}'' + J_{N,3}'+ J_{N,3}'' + J_{N,4} - (V_{N,1} + V_{N,2} + V_{N,3})
	\\
	& =  J_{N,1}+J_{N,2}'+   \underbrace{(J_{N,2}'' - V_{N,1})}_{ \text{denoted $\widehat{J}_{N,2}$}}+ J_{N,3}' + \underbrace{(J_{N,3}'' - V_{N,2})}_{\text{denoted $\widehat{J}_{N,3}$}}+ \underbrace{(J_{N,4} - V_{N,3})}_{\text{denoted $\widehat{J}_{N,4}$}} 
	\\
	& = J_{N,1}+J_{N,2}'+   \widehat{J}_{N,2} + J_{N,3}' +  \widehat{J}_{N,3} + \widehat{J}_{N,4},
	\end{split}
\end{align}
where the explicit expressions of $J_{N,1}$, $J_{N,2}'$ and $J_{N,3}'$ are given  in \S \ref{sec-square-J}, and by direct computation, the explicit expressions of $\widehat{J}_{N,2}$, $\widehat{J}_{N,3}$ and $\widehat{J}_{N,4}$ are as follows:
\begin{align*}
	\widehat{J}_{N,2}=\int_{\mathcal{U}_N(z)^2}T_z(x_1)^{2s}T_z(x_2)^{2s}\mathcal{J}_2(x_1,x_2)\prod_{k=1}^{2}\mathrm{d}\mu(x_k),
\end{align*}
\begin{align*}
	\widehat{J}_{N,3}=2\int_{\mathcal{U}_N(z)^3}T_z(x_1)^{2s}T_z(x_2)^{s}T_z(x_3)^{s}\mathcal{J}_3(x_1,x_2,x_3)\prod_{k=1}^{3}\mathrm{d}\mu(x_k),
\end{align*}
\begin{align*}
	\widehat{J}_{N,4}=\int_{\mathcal{U}_N(z)^4}\Big(\prod_{k=1}^{4}T_z(x_k)^{s}\Big)\mathcal{J}_4(x_1,x_2,x_3,x_4)\prod_{k=1}^{4}\mathrm{d}\mu(x_k),
\end{align*}
with
\begin{align*}
	\mathcal{J}_2(x_1,x_2)=[K_{11}K_{22}+K_{12}^2+K_{12}K_{21}]\det[K_{ij}]_{1\leq i,j\leq2}-K_{11}^2K_{22}^2,
\end{align*}
\begin{align*}
	\mathcal{J}_3(x_1,x_2,x_3)=K_{11}K_{23}\big(\det[K_{ij}]_{1\leq i,j\leq3}-K_{11}\det[K_{ij}]_{2\leq i,j\leq3}\big),
\end{align*}
\begin{align*}
	\mathcal{J}_4(x_1,x_2,x_3,x_4)=K_{12}K_{34}\big(\det[K_{ij}]_{1\leq i,j\leq4}-\det[K_{ij}]_{1\leq i,j\leq2}\det[K_{ij}]_{3\leq i,j\leq4}\big).
\end{align*}

\begin{remark*} For obtaining the up-estimate of $\Var(S_N)$,  we are going to give up-estimate for all the terms $J_{N,1}$, $J_{N,2}'$,  $\widehat{J}_{N,2}$, $J_{N,3}'$, $\widehat{J}_{N,3}$ and $\widehat{J}_{N,4}$.  Note that it is crucial for us to  give the up-estimate for  $\widehat{J}_{N,2} = J_{N,2}''-V_{N,1}$, but not direct up-estimate for both terms $J_{N,2}''$ and $V_{N,1}$. More precisely, the following naive up-estimate is not sufficient for our purpose$:$ 
\begin{align*}
|J_{N,2}''-V_{N,1}| \le  |J_{N,2}''| + |V_{N,1}|. 
\end{align*}
Our proof relies on the following cancellation of main parts of $J_{N,2}''$ and $V_{N,1}$$:$   
\begin{align*}
J_{N,2}''-V_{N,1} = o( \min \{J_{N,2}'', V_{N,1}\}).
\end{align*}
Similarly, major cancellations arise in the terms $J_{N,3}''-V_{N,2}$ and $J_{N,4}-V_{N,3}$$:$  
\begin{align*}
J_{N,3}''-V_{N,2} =o (\min\{ J_{N,3}'', V_{N,2}\}) \an   J_{N,4}-V_{N,3} = o (\min\{  J_{N,4},   V_{N,3} \}). 
\end{align*}
\end{remark*}

\subsection{The estimate for $J_{N,1}$}
\begin{lemma}\label{lem-R1N}
	There exists $c_1(s)>0$, such that for large enough $N$, we have
	\begin{align*}
	|J_{N,1}|\leq c_1(s)\frac{(1+|z|)^4}{(1-|z|)^4}e^{(3-2s)N}.
\end{align*}
\end{lemma}

\begin{proof}
Since
\begin{align*}
		J_{N,1}=\int_{\mathcal{U}_N(z)}T_z(x_1)^{4s}K_{11}^3\mathrm{d}\mu(x_1),
\end{align*}
by setting $w_1=\varphi_z(x_1)$, we have
\begin{align*}
		J_{N,1}&=\frac{1}{(1-|z|^2)^4}\int_{D(0, r_N)}\frac{(1-|w_1|)^{4s-6}}{(1+|w_1|)^{4s+6}}|1-\overline{z}w_1|^8\mathrm{d}\mu(w_1)\\
		&\leq\frac{(1+|z|)^4}{(1-|z|)^4}\int_{D(0, r_N)}(1-|w_1|)^{4s-6}\mathrm{d}\mu(w_1).
\end{align*}
Then, by using the polar coordinates,
\begin{align*}
		|J_{N,1}|\leq\frac{2(1+|z|)^4}{(1-|z|)^4}\int_{0}^{r_N}(1-r)^{4s-6}\mathrm{d}r.
\end{align*}
It is easy to verify that there exists $A_1(s)>0$, such that when $N$ is large enough,
\begin{align*}
		\int_{0}^{r_N}(1-r)^{4s-6}\mathrm{d}r\leq A_1(s)e^{(3-2s)N}.
\end{align*}
Therefore, there exists $c_1(s)>0$, such that for large enough $N$,
\begin{align*}
		|J_{N,1}|\leq c_1(s)\frac{(1+|z|)^4}{(1-|z|)^4}e^{(3-2s)N}.
\end{align*}

This completes the proof of Lemma~\ref{lem-R1N}.
\end{proof}

\subsection{The estimate for $J_{N,2}'$}
\begin{lemma}\label{lem-R2N}
	There exists $c_2'(s)>0$, such that for large enough $N$, we have
	\begin{align*}
	|J_{N,2}'|\leq c_2'(s)\frac{(1+|z|)^4}{(1-|z|)^4}e^{(3-2s)N}.
\end{align*}
\end{lemma}

\begin{proof}
By the Hadamard inequality (see, e.g., \cite[Formula~(4.2.15)]{HKPV} or \cite[Theorem~13.5.5]{Mi})
\begin{align*}
		\det[K_{ij}]_{1\leq i,j\leq2}\leq K_{11}K_{22},
\end{align*}
we get
\begin{align*}
		|J_{N,2}'|\leq4\int_{\mathcal{U}_N(z)^2}T_z(x_1)^{3s}T_z(x_2)^{s}K_{11}^2K_{22}|K_{12}|\prod_{k=1}^{2}\mathrm{d}\mu(x_k).
\end{align*}
By setting $w_k=\varphi_z(x_k)$, $k=1,2$, we have
\begin{align*}
		|J_{N,2}'|&\leq\frac{4}{(1-|z|^2)^4}\int_{D(0, r_N)^2}\frac{(1-|w_1|)^{3s-4}}{(1+|w_1|)^{3s+4}}\frac{(1-|w_2|)^{s-2}}{(1+|w_2|)^{s+2}}\frac{|1-\overline{z}w_1|^6|1-\overline{z}w_2|^2}{|1-w_1\overline{w}_2|^2}\prod_{k=1}^{2}\mathrm{d}\mu(w_k)\\
		&\leq\frac{4(1+|z|)^4}{(1-|z|)^4}\int_{D(0, r_N)^2}\frac{(1-|w_1|)^{3s-4}(1-|w_2|)^{s-2}}{|1-w_1\overline{w}_2|^2}\prod_{k=1}^{2}\mathrm{d}\mu(w_k).
\end{align*}
Then, by using the polar coordinates,
\begin{align*}
		|J_{N,2}'|\leq\frac{16(1+|z|)^4}{(1-|z|)^4}\int_{0}^{r_N}\int_{0}^{r_N}\frac{(1-r_1)^{3s-4}(1-r_2)^{s-2}r_1r_2}{1-r_1^2r_2^2}\mathrm{d}r_1\mathrm{d}r_2.
\end{align*}
Noticing that
\begin{align}\label{inequ-r1r2}
		1-r_1^2r_2^2\geq1-r_1r_2=(1-r_1r_2)^{\frac{1}{2}}(1-r_1r_2)^{\frac{1}{2}}\geq(1-r_1)^{\frac{1}{2}}(1-r_2)^{\frac{1}{2}},
\end{align}
we get
\begin{align*}
		|J_{N,2}'|\leq\frac{16(1+|z|)^4}{(1-|z|)^4}\int_{0}^{r_N}(1-r)^{3s-\frac{9}{2}}\mathrm{d}r\int_{0}^{r_N}(1-r)^{s-\frac{5}{2}}\mathrm{d}r.
\end{align*}
It is easy to verify that there exist $A_2(s),A_3(s)>0$, such that when $N$ is large enough,
\begin{align*}
		\int_{0}^{r_N}(1-r)^{3s-\frac{9}{2}}\mathrm{d}r\leq A_2(s)e^{(\frac{3}{2}-s)N},\quad \int_{0}^{r_N}(1-r)^{s-\frac{5}{2}}\mathrm{d}r\leq A_3(s)e^{(\frac{3}{2}-s)N}.
\end{align*}
Therefore, there exists $c_2'(s)>0$, such that for large enough $N$,
\begin{align*}
		|J_{N,2}'|\leq c_2'(s)\frac{(1+|z|)^4}{(1-|z|)^4}e^{(3-2s)N}.
\end{align*}
	
This completes the proof of Lemma~\ref{lem-R2N}.
\end{proof}

\subsection{The estimate for $\widehat{J}_{N,2}$}
\begin{lemma}\label{lem-R3N}
	There exists $\widehat{c}_2(s)>0$, such that for large enough $N$, we have
	\begin{align*}
	|\widehat{J}_{N,2}|\leq \widehat{c}_2(s)\frac{(1+|z|)^4}{(1-|z|)^4}e^{(3-2s)N}.
\end{align*}
\end{lemma}

\begin{proof}
By a direct calculation, we have
\begin{align*}
		\mathcal{J}_2(x_1,x_2)=K_{11}K_{22}K_{12}^2-K_{12}^3K_{21}-K_{12}^2K_{21}^2.
\end{align*}
Noticing that
\begin{align*}
		|K_{12}|=|K_{21}|\leq K_{11}^{\frac{1}{2}}K_{22}^{\frac{1}{2}},
\end{align*}
we get
\begin{align*}
		|\mathcal{J}_2(x_1,x_2)|\leq3K_{11}^{\frac{3}{2}}K_{22}^{\frac{3}{2}}|K_{12}|.
\end{align*}
Hence
\begin{align*}
		|\widehat{J}_{N,2}|\leq3\int_{\mathcal{U}_N(z)^2}T_z(x_1)^{2s}T_z(x_2)^{2s}K_{11}^{\frac{3}{2}}K_{22}^{\frac{3}{2}}|K_{12}|\prod_{k=1}^{2}\mathrm{d}\mu(x_k).
\end{align*}
By setting $w_k=\varphi_z(x_k)$, $k=1,2$, we have
\begin{align*}
		|\widehat{J}_{N,2}|&\leq\frac{3}{(1-|z|^2)^4}\int_{D(0, r_N)^2}\frac{(1-|w_1|)^{2s-3}}{(1+|w_1|)^{2s+3}}\frac{(1-|w_2|)^{2s-3}}{(1+|w_2|)^{2s+3}}\frac{|1-\overline{z}w_1|^4|1-\overline{z}w_2|^4}{|1-w_1\overline{w}_2|^2}\prod_{k=1}^{2}\mathrm{d}\mu(w_k)\\
		&\leq\frac{3(1+|z|)^4}{(1-|z|)^4}\int_{D(0, r_N)^2}\frac{(1-|w_1|)^{2s-3}(1-|w_2|)^{2s-3}}{|1-w_1\overline{w}_2|^2}\prod_{k=1}^{2}\mathrm{d}\mu(w_k).
\end{align*}
Then, by using the polar coordinates,
\begin{align*}
		|\widehat{J}_{N,2}|\leq\frac{12(1+|z|)^4}{(1-|z|)^4}\int_{0}^{r_N}\int_{0}^{r_N}\frac{(1-r_1)^{2s-3}(1-r_2)^{2s-3}r_1r_2}{1-r_1^2r_2^2}\mathrm{d}r_1\mathrm{d}r_2.
\end{align*}
Thus by the inequality \eqref{inequ-r1r2}, we get
\begin{align*}
		|\widehat{J}_{N,2}|\leq\frac{12(1+|z|)^4}{(1-|z|)^4}\Big(\int_{0}^{r_N}(1-r)^{2s-\frac{7}{2}}\mathrm{d}r\Big)^2.
\end{align*}
It is easy to verify that there exists $A_4(s)>0$, such that when $N$ is large enough,
\begin{align*}
		\int_{0}^{r_N}(1-r)^{2s-\frac{7}{2}}\mathrm{d}r\leq A_4(s)e^{(\frac{3}{2}-s)N}.
\end{align*}
Therefore, there exists $\widehat{c}_2(s)>0$, such that for large enough $N$,
\begin{align*}
		|\widehat{J}_{N,2}|\leq \widehat{c}_2(s)\frac{(1+|z|)^4}{(1-|z|)^4}e^{(3-2s)N}.
\end{align*}
	
This completes the proof of Lemma~\ref{lem-R3N}.
\end{proof}

\subsection{The estimate for $J_{N,3}'$}
\begin{lemma}\label{lem-R4N}
	There exists $c_3'(s)>0$, such that for large enough $N$, we have
	\begin{align*}
	|J_{N,3}'|\leq c_3'(s)\frac{(1+|z|)^4}{(1-|z|)^4}e^{(3-2s)N}.
\end{align*}
\end{lemma}

\begin{proof}
By the Hadamard inequality (see, e.g., \cite[Formula~(4.2.15)]{HKPV} or \cite[Theorem~13.5.5]{Mi})
\begin{align*}
		\det[K_{ij}]_{1\leq i,j\leq3}\leq K_{11}K_{22}K_{33},
\end{align*}
we get
\begin{align*}
		|J_{N,3}'|\leq4\int_{\mathcal{U}_N(z)^3}T_z(x_1)^{2s}T_z(x_2)^{s}T_z(x_3)^{s}K_{11}K_{22}K_{33}|K_{12}||K_{13}|\prod_{k=1}^{3}\mathrm{d}\mu(x_k).
\end{align*}
By setting $w_k=\varphi_z(x_k)$, $k=1,2,3$, we have
\begin{align*}
		|J_{N,3}'|\leq\frac{4}{(1-|z|)^4}\int_{D(0, r_N)^3}\mathcal{M}(w_1,w_2,w_3)\prod_{k=1}^{3}\mathrm{d}\mu(w_k),
\end{align*}
where
\begin{align*}
	\mathcal{M}(w_1,w_2,w_3)=\frac{(1-|w_1|)^{2s-2}}{(1+|w_1|)^{2s+2}}\frac{(1-|w_2|)^{s-2}}{(1+|w_2|)^{s+2}}\frac{(1-|w_3|)^{s-2}}{(1+|w_3|)^{s+2}}\frac{|1-\overline{z}w_1|^4|1-\overline{z}w_2|^2|1-\overline{z}w_3|^2}{|1-w_1\overline{w}_2|^2|1-w_1\overline{w}_3|^2}.
\end{align*}
This yields that
\begin{align*}
	|J_{N,3}'|\leq\frac{4(1+|z|)^4}{(1-|z|)^4}\int_{D(0, r_N)^3}\frac{(1-|w_1|)^{2s-2}(1-|w_2|)^{s-2}(1-|w_3|)^{s-2}}{|1-w_1\overline{w}_2|^2|1-w_1\overline{w}_3|^2}\prod_{k=1}^{3}\mathrm{d}\mu(w_k).
\end{align*}
Then, by using the polar coordinates,
\begin{align*}
		|J_{N,3}'|\leq\frac{32(1+|z|)^4}{(1-|z|)^4}\int_{0}^{r_N}\int_{0}^{r_N}\int_{0}^{r_N}\frac{(1-r_1)^{2s-2}(1-r_2)^{s-2}(1-r_3)^{s-2}r_1r_2r_3}{(1-r_1^2r_2^2)(1-r_1^2r_3^2)}\mathrm{d}r_1\mathrm{d}r_2\mathrm{d}r_3.
\end{align*}
By the same inequality as \eqref{inequ-r1r2}, for any $k,j\in\{1,2,3\}$,
\begin{align}\label{inequ-r1r2r3}
		1-r_k^2r_j^2\geq(1-r_k)^{\frac{1}{2}}(1-r_j)^{\frac{1}{2}},
\end{align}
so we get
\begin{align*}
		|J_{N,3}'|\leq\frac{32(1+|z|)^4}{(1-|z|)^4}\int_{0}^{r_N}(1-r)^{2s-3}\mathrm{d}r\Big(\int_{0}^{r_N}(1-r)^{s-\frac{5}{2}}\mathrm{d}r\Big)^2.
\end{align*}
It is easy to verify that there exist $A_5(s),A_6(s)>0$, such that when $N$ is large enough,
\begin{align*}
		\int_{0}^{r_N}(1-r)^{2s-3}\mathrm{d}r\leq A_5(s),\quad \int_{0}^{r_N}(1-r)^{s-\frac{5}{2}}\mathrm{d}r\leq A_6(s)e^{(\frac{3}{2}-s)N}.
\end{align*}
Therefore, there exists $c_3'(s)>0$, such that for large enough $N$,
\begin{align*}
		|J_{N,3}'|\leq c_3'(s)\frac{(1+|z|)^4}{(1-|z|)^4}e^{(3-2s)N}.
\end{align*}
	
This completes the proof of Lemma~\ref{lem-R4N}.
\end{proof}

\subsection{The estimate for $\widehat{J}_{N,3}$}
\begin{lemma}\label{lem-R5N}
	There exists $\widehat{c}_3(s)>0$, such that for large enough $N$, we have
	\begin{align*}
	|\widehat{J}_{N,3}|\leq \widehat{c}_3(s)\frac{(1+|z|)^4}{(1-|z|)^4}e^{(3-2s)N}.
	\end{align*}
\end{lemma}

\begin{proof}
A direct calculation gives that
\begin{align*}
		\det[K_{ij}]_{1\leq i,j\leq3}-K_{11}\det[K_{ij}]_{2\leq i,j\leq3}=-K_{12}K_{21}K_{33}+K_{12}K_{23}K_{31}+K_{13}K_{21}K_{32}-K_{13}K_{22}K_{31}.
\end{align*}
Noticing that for any $k,j\in\{1,2,3\}$,
\begin{align*}
		|K_{kj}|=|K_{jk}|\leq K_{kk}^{\frac{1}{2}}K_{jj}^{\frac{1}{2}},
\end{align*}
we get
\begin{align*}
		|\mathcal{J}_3(x_1,x_2,x_3)|\leq2K_{11}^{\frac{3}{2}}K_{22}^{\frac{1}{2}}K_{33}|K_{12}||K_{23}|+2K_{11}^{\frac{3}{2}}K_{33}^{\frac{1}{2}}K_{22}|K_{13}||K_{32}|.
\end{align*}
Hence
\begin{align*}
		|\widehat{J}_{N,3}|\leq8\int_{\mathcal{U}_N(z)^3}T_z(x_1)^{2s}T_z(x_2)^{s}T_z(x_3)^{s}K_{11}^{\frac{3}{2}}K_{22}^{\frac{1}{2}}K_{33}|K_{12}||K_{23}|\prod_{k=1}^{3}\mathrm{d}\mu(x_k).
\end{align*}
By setting $w_k=\varphi_z(x_k)$, $k=1,2,3$, we have
\begin{align*}
		|\widehat{J}_{N,3}|\leq\frac{8}{(1-|z|)^4}\int_{D(0, r_N)^3}\mathcal{N}(w_1,w_2,w_3)\prod_{k=1}^{3}\mathrm{d}\mu(w_k),
\end{align*}
where
\begin{align*}
	\mathcal{N}(w_1,w_2,w_3)=\frac{(1-|w_1|)^{2s-3}}{(1+|w_1|)^{2s+3}}\frac{(1-|w_2|)^{s-1}}{(1+|w_2|)^{s+1}}\frac{(1-|w_3|)^{s-2}}{(1+|w_3|)^{s+2}}\frac{|1-\overline{z}w_1|^4|1-\overline{z}w_2|^2|1-\overline{z}w_3|^2}{|1-w_1\overline{w}_2|^2|1-w_2\overline{w}_3|^2}.
\end{align*}
This yields that
\begin{align*}
	|\widehat{J}_{N,3}|\leq\frac{8(1+|z|)^4}{(1-|z|)^4}\int_{D(0, r_N)^3}\frac{(1-|w_1|)^{2s-3}(1-|w_2|)^{s-1}(1-|w_3|)^{s-2}}{|1-w_1\overline{w}_2|^2|1-w_2\overline{w}_3|^2}\prod_{k=1}^{3}\mathrm{d}\mu(w_k).
\end{align*}
Then, by using the polar coordinates,
\begin{align*}
		|\widehat{J}_{N,3}|\leq\frac{64(1+|z|)^4}{(1-|z|)^4}\int_{0}^{r_N}\int_{0}^{r_N}\int_{0}^{r_N}\frac{(1-r_1)^{2s-3}(1-r_2)^{s-1}(1-r_3)^{s-2}r_1r_2r_3}{(1-r_1^2r_2^2)(1-r_2^2r_3^2)}\mathrm{d}r_1\mathrm{d}r_2\mathrm{d}r_3.
\end{align*}
Thus by the inequality \eqref{inequ-r1r2r3}, we get
\begin{align*}
		|\widehat{J}_{N,3}|\leq\frac{64(1+|z|)^4}{(1-|z|)^4}\int_{0}^{r_N}(1-r)^{2s-\frac{7}{2}}\mathrm{d}r\int_{0}^{r_N}(1-r)^{s-2}\mathrm{d}r\int_{0}^{r_N}(1-r)^{s-\frac{5}{2}}\mathrm{d}r.
\end{align*}
It is easy to verify that there exist $A_7(s),A_8(s),A_9(s)>0$, such that when $N$ is large enough,
\begin{align*}
		\int_{0}^{r_N}(1-r)^{2s-\frac{7}{2}}\mathrm{d}r\leq A_7(s)e^{(\frac{3}{2}-s)N},\quad\int_{0}^{r_N}(1-r)^{s-2}\mathrm{d}r\leq A_8(s)
\end{align*}
and
\begin{align*}
		\int_{0}^{r_N}(1-r)^{s-\frac{5}{2}}\mathrm{d}r\mathrm{d}r\leq A_9(s)e^{(\frac{3}{2}-s)N}.
\end{align*}
Therefore, there exists $\widehat{c}_3(s)>0$, such that for large enough $N$,
\begin{align*}
		|\widehat{J}_{N,3}|\leq \widehat{c}_3(s)\frac{(1+|z|)^4}{(1-|z|)^4}e^{(3-2s)N}.
\end{align*}

This completes the proof of Lemma~\ref{lem-R5N}.
\end{proof}

\subsection{The estimate for $\widehat{J}_{N,4}$}
\begin{lemma}\label{lem-R6N}
	There exists $\widehat{c}_4(s)>0$, such that for large enough $N$, we have
	\begin{align*}
	|\widehat{J}_{N,4}|\leq \widehat{c}_4(s)\frac{(1+|z|)^4}{(1-|z|)^4}e^{(3-2s)N}.
	\end{align*}
\end{lemma}

\begin{proof}
A direct calculation gives that
\begin{align*}
		\det[K_{ij}]_{1\leq i,j\leq4}-\det[K_{ij}]_{1\leq i,j\leq2}\det[K_{ij}]_{3\leq i,j\leq4}=\sum_{k=1}^{4}\mathcal{W}_k(x_1,x_2,x_3,x_4)
\end{align*}
with
\begin{align*}
		\mathcal{W}_1(x_1,x_2,x_3,x_4)=-K_{11}K_{23}K_{32}K_{44}+K_{11}K_{23}K_{34}K_{42}+K_{12}K_{23}K_{31}K_{44}-K_{12}K_{23}K_{34}K_{41},
\end{align*}
\begin{align*}
		\mathcal{W}_2(x_1,x_2,x_3,x_4)=K_{11}K_{24}K_{32}K_{43}-K_{11}K_{24}K_{33}K_{42}-K_{12}K_{24}K_{31}K_{43}+K_{12}K_{24}K_{33}K_{41},
\end{align*}
and
\begin{align*}
		\mathcal{W}_3(x_1,x_2,x_3,x_4)&=K_{13}K_{21}K_{32}K_{44}-K_{13}K_{21}K_{34}K_{42}-K_{13}K_{22}K_{31}K_{44}\\
		&\quad+K_{13}K_{22}K_{34}K_{41}+K_{13}K_{24}K_{31}K_{42}-K_{13}K_{24}K_{32}K_{41},
\end{align*}
\begin{align*}
		\mathcal{W}_4(x_1,x_2,x_3,x_4)&=-K_{14}K_{21}K_{32}K_{43}+K_{14}K_{21}K_{33}K_{42}+K_{14}K_{22}K_{31}K_{43}\\
		&\quad-K_{14}K_{22}K_{33}K_{41}-K_{14}K_{23}K_{31}K_{42}+K_{14}K_{23}K_{32}K_{41}.
\end{align*}
Noticing that for any $k,j\in\{1,2,3,4\}$,
\begin{align*}
		|K_{kj}|=|K_{jk}|\leq K_{kk}^{\frac{1}{2}}K_{jj}^{\frac{1}{2}},
\end{align*}
we get
\begin{align*}
		|\mathcal{J}_4(x_1,x_2,x_3,x_4)|&\leq4K_{11}K_{22}^{\frac{1}{2}}K_{33}^{\frac{1}{2}}K_{44}|K_{12}||K_{23}||K_{34}|+4K_{11}K_{22}^{\frac{1}{2}}K_{44}^{\frac{1}{2}}K_{33}|K_{12}||K_{24}||K_{43}|\\
		&\quad+6K_{22}K_{11}^{\frac{1}{2}}K_{33}^{\frac{1}{2}}K_{44}|K_{21}||K_{13}||K_{34}|+6K_{22}K_{11}^{\frac{1}{2}}K_{44}^{\frac{1}{2}}K_{33}|K_{21}||K_{14}||K_{43}|.
\end{align*}
Hence
\begin{align*}
		|\widehat{J}_{N,4}|&\leq20\int_{\mathcal{U}_N(z)^4}  \Big(\prod_{k=1}^4 T_z(x_j)^{s} \Big) K_{11}K_{22}^{\frac{1}{2}}K_{33}^{\frac{1}{2}}K_{44}|K_{12}||K_{23}||K_{34}|\prod_{k=1}^4\mathrm{d}\mu(x_k).
\end{align*}
By setting $w_k=\varphi_z(x_k)$, $k=1,2,3,4$, we have
\begin{align*}
		|\widehat{J}_{N,4}|\leq\frac{20}{(1-|z|^2)^4}\int_{D(0, r_N)^4}\mathcal{F}_1(w_1,w_2,w_3,w_4)\mathcal{F}_2(w_1,w_2,w_3,w_4)\prod_{k=1}^{4}\mathrm{d}\mu(w_k),
\end{align*}
where
\begin{align*}
	\mathcal{F}_1(w_1,w_2,w_3,w_4)=\frac{(1-|w_1|)^{s-2}}{(1+|w_1|)^{s+2}}\frac{(1-|w_2|)^{s-1}}{(1+|w_2|)^{s+1}}\frac{(1-|w_3|)^{s-1}}{(1+|w_3|)^{s+1}}\frac{(1-|w_4|)^{s-2}}{(1+|w_4|)^{s+2}}
\end{align*}
and
\begin{align*}
	\mathcal{F}_2(w_1,w_2,w_3,w_4)=\frac{|1-\overline{z}w_1|^2|1-\overline{z}w_2|^2|1-\overline{z}w_3|^2|1-\overline{z}w_4|^2}{|1-w_1\overline{w}_2|^2|1-w_2\overline{w}_3|^2|1-w_3\overline{w}_4|^2}.
\end{align*}
This yields that
\begin{align*}
	|\widehat{J}_{N,4}|\leq\frac{20(1+|z|)^4}{(1-|z|)^4}\int_{D(0, r_N)^4}\mathcal{G}(w_1,w_2,w_3,w_4)\prod_{k=1}^{4}\mathrm{d}\mu(w_k),
\end{align*}
where
\begin{align*}
		\mathcal{G}(w_1,w_2,w_3,w_4)=\frac{(1-|w_1|)^{s-2}(1-|w_2|)^{s-1}(1-|w_3|)^{s-1}(1-|w_4|)^{s-2}}{|1-w_1\overline{w}_2|^2|1-w_2\overline{w}_3|^2|1-w_3\overline{w}_4|^2}.
\end{align*}
Then, by using the polar coordinates,
\begin{align*}
		|\widehat{J}_{N,4}|\leq\frac{320(1+|z|)^4}{(1-|z|)^4}\int_{0}^{r_N}\int_{0}^{r_N}\int_{0}^{r_N}\int_{0}^{r_N}\mathcal{H}(r_1,r_2,r_3,r_4)\mathrm{d}r_1\mathrm{d}r_2\mathrm{d}r_3\mathrm{d}r_4,
\end{align*}
where
\begin{align*}
		\mathcal{H}(r_1,r_2,r_3,r_4)=\frac{(1-r_1)^{s-2}(1-r_2)^{s-1}(1-r_3)^{s-1}(1-r_4)^{s-2}r_1r_2r_3r_4}{(1-r_1^2r_2^2)(1-r_2^2r_3^2)(1-r_3^2r_4^2)}.
\end{align*}
By the same inequality as \eqref{inequ-r1r2}, for any $k,j\in\{1,2,3,4\}$,
\begin{align*}
		1-r_k^2r_j^2\geq(1-r_k)^{\frac{1}{2}}(1-r_j)^{\frac{1}{2}},
\end{align*}
so we get
\begin{align*}
		|\widehat{J}_{N,4}|\leq\frac{320(1+|z|)^4}{(1-|z|)^4}\Big(\int_{0}^{r_N}(1-r)^{s-\frac{5}{2}}\mathrm{d}r\Big)^2\Big(\int_{0}^{r_N}(1-r)^{s-2}\mathrm{d}r\Big)^2.
\end{align*}
It is easy to verify that there exist $A_{10}(s),A_{11}(s)>0$, such that when $N$ is large enough,
\begin{align*}
		\int_{0}^{r_N}(1-r)^{s-\frac{5}{2}}\mathrm{d}r\leq A_{10}(s)e^{(\frac{3}{2}-s)N},\quad \int_{0}^{r_N}(1-r)^{s-2}\mathrm{d}r\leq A_{11}(s).
\end{align*}
Therefore, there exists $\widehat{c}_4(s)>0$, such that for large enough $N$,
\begin{align*}
		|\widehat{J}_{N,4}|\leq \widehat{c}_4(s)\frac{(1+|z|)^4}{(1-|z|)^4}e^{(3-2s)N}.
\end{align*}
	
This completes the proof of Lemma~\ref{lem-R6N}.
\end{proof}

\subsection{The conclusion for $\Var(S_N)$: the proof of Lemma~\ref{lem-var}}
\begin{proof}[Proof of Lemma~\ref{lem-var}]
	By \eqref{dec-var}, we obtain Lemma~\ref{lem-var} from Lemmas~\ref{lem-R1N}-\ref{lem-R6N} directly. This completes the whole proof of Lemma~\ref{lem-var}.
\end{proof}

\end{document}